\documentclass[11pt,twoside]{article}
\usepackage{amsmath,amsfonts,amssymb,theorem,amscd}

\numberwithin{equation}{section}

\newtheorem{theorem}{Theorem}[section]
\newtheorem{lemma}{Lemma}[section]
\newtheorem{proposition}{Proposition}[section]

\newtheorem{definition}{Definition}[section]

\usepackage[usenames,dvipsnames]{pstricks}
\usepackage{epsfig}
\usepackage{pst-grad} 
\usepackage{pst-plot} 

\def\XXint#1#2#3{{\setbox0=\hbox{$#1{#2#3}{\int}$}
     \vcenter{\hbox{$#2#3$}}\kern-.5\wd0}}

\newcommand{\qed}{\hfill\fbox{}\par\vspace{.2cm}}

\newcommand{\norm}[1]{\left\Vert #1 \right\Vert}
\newcommand{\abs}[1]{\left| #1 \right|}

\newenvironment{proof}{{\bf Proof.}} {\hfill\qed}

\newcommand{\R}{{ \mathbb{R}  }}
\newcommand{\T}{{ \mathbb{T}  }}

\begin{document}
\renewcommand{\thesection}{\arabic{section}}

\newcommand{\pr}{\partial}
\newcommand{\nl}{\vskip 1pc}
\newcommand{\co}{\mbox{co}}
\newcommand{\bbr}{\mathbb R}

\title{A Weighted Regularity Criterion for Suitable Weak Solutions of Incompressible Non-Newtonian Fluids}

\author{Jae-Myoung Kim}
\date{}
\maketitle \centerline{\textit{Department of Mathematics Education}}
\centerline{\textit{Gyeongkuk  National University, Andong, Republic of
Korea}} \centerline{\textit{jmkim02@gknu.ac.kr}}

\begin{abstract}
It establishes a regularity criterion for non-Newtonian fluids in $\mathbb{R}^3$ in terms of the weighted gradient of the velocity field, based on the
Caffarelli--Kohn--Nirenberg inequality.
\end{abstract}

\bigskip
\bigskip

\vspace{-5mm} \noindent {2000 AMS Subject Classification}: 76A05,
76D03, 49N60
\newline {Keywords} : non-Newtonian fluid; weighted regularity criteria, Caffarelli-Kohn-Nirenberg inequality.

\section{Introduction}

In this paper, we study the non-Newtonian fluid of shear-thinning type:
\begin{equation}\label{NNS}
\left\{
\begin{array}{ll}
\displaystyle u_t- \nabla \cdot S +(u \cdot
\nabla) u +\nabla \pi=0,\\
\vspace{-3mm}\\
\displaystyle \text{div} \ u =0,
\end{array}\right.
\quad \mbox{ in } \,\,Q_T:=\R^3\times (0,\, T),
\end{equation}
where $u:\R^3\times (0,\, T)\rightarrow\R^3$ is the flow velocity
vector and $\pi:\R^3\times (0,\, T)\rightarrow\R$ is the pressure.
Also, $S=(S_{ij})_{i,j=1,2,3}$ is the stress tensor depending on the
deformation tensor $Du = (\nabla u + \nabla u^{T})/2$. We consider
the initial value problem of \eqref{NNS}, which requires initial
conditions
\begin{equation}\label{NNS-30}
u(x,0)=u_0(x).
\end{equation}
We assume that the initial data $u_0(x) \in L^2(\Omega)$ hold the
incompressibility, i.e. $ \text{div} \ u_0(x)=0$. To motivate the
conditions on the stress tensor $S$ is called "power law fluids". This paper deals with the following forms
\begin{equation}\label{S-example-10}
S(Du) = (\mu_0+\mu_1|Du|^{q-2})Du, \quad 1 < q < +\infty.
\end{equation}
where $\mu_0>0$ and $\mu_1>0$ are constants (see e.g. \cite{L69},
\cite{B}).

We review the known results regarding existence of solutions related to our result. 
For $\mu_0\ge 0$ and $\mu_1>0$, the existence of weak
solutions for $\frac{3n+2}{n+2}\leq q$ was firstly established in \cite{L69}, and later,
in \cite{D-R-W10}, the range of $q$ was extended up to $
\frac{2n}{n+2}<q$.
If $\mu_0> 0$ and  $\frac{n+2}{2}\le q$, it was shown that  the weak solution is
unique (see e.g. \cite{M-N-R-R} and compare to \cite{Lion69}). 
M\'alek et.al, in \cite{M-N-R-R} proved that the unique strong
solution exists globally in time for $q \geq \frac{3n+2}{n+2}$ in
the case of periodic domains of $\R^n$, $n\ge 3$ (see \cite{P} for the case of the whole space).  

We report the known results related to the regularity criteria of a weak solution in Navier-stokes equations of non-Newontian type. For a thinning fluid, that is, $\frac{8}{5} < p < 2$, Bae et al. \cite{BCK99} obtained the
regularity criterion of Serrin type for Ostwald-de Waele flows (i.e. a class of shear thinning fluids):
\[
\norm{u}_{L^{\alpha,\beta}_{x,t}(Q_T)}:=\norm{\norm{u}_{L^\alpha(\R^3)}}_{L^\beta(0,T)}<\infty,
\quad \frac{3}{\alpha}+
\frac{5p-6}{2\beta}\leq  \frac{5p-8}{2},\quad \alpha
> \frac{6}{5p-8}.
\]
For a point $z=(x,t)\in \R^3\times (0,T)$, we denote
$B_\rho(x)=\{y\in\R^3: \abs{y-x}<\rho\}$,
\[
Q_\rho(z)=B_\rho(x)\times (t-\rho^2,t),\quad \rho<\sqrt{t}.
\]
On the other hand, the authors \cite{K-M-S02} showed that the global-in-time existence of a solution to the problem \eqref{NNS} considered such that $u\in C^{1,\alpha} (\T^2\times (0,T))$, $ 0<\alpha< 1$, provided $q>\frac{4}{3}$ and some conditions on the function the external force $f$ are fulfilled.

%

In the present paper, we give regularity conditions of the weighted type
for a weak solution to the 3D the non-Newtonian fluid of shear-thinning type which is extensions of
results in \cite{Yong09}. For a proof, we use the energy estimate based on the
Caffarelli-Kohn-Nirenberg and Stein's inequalities. After
then, using $\epsilon$-regularity criteria for the suitable weak solution, we
show a weak solution is semi-regular.

%

Next, we are to state the second part of our main result.

\begin{theorem}\label{thm1}
Let $q>2$. Assume that $u_0 \in H^1(\R^3)$ with $\nabla
\cdot u_0 = 0$ satisfies
\[
\sup_{x_0\in \R^3}\|x- x_0|^{-\frac{1}{2}}u_0\|_{L^2} < \infty, 
\]and 
if a weak solution
$u$ satisfies
\[\sup_{x_0\in \R^3}\||x- x_0|^{\beta}\nabla u\|_{L_{x,t}^{r,s}}<
\infty,\quad \sup_{x_0\in R^3}\int_0^T \|x-
x_0|^{-\frac{4}{3}}|\nabla u(\cdot,\tau)|^{q-1}\|^{\frac{3}{2}}_{L^{\frac{3}{2}}}\, d\tau<
\infty,
\]
with
\[
\frac{3}{r}+ \frac{5q-6}{2s} = \frac{5q-8}{2}+1- \beta, \quad 1 <s < \infty,
\]
\[
 \frac{3}{ \frac{5q-8}{2}+1- \beta}< r < \infty, \quad -1 \leq  \beta <\frac{5q-6}{2}.
\]
Then  the weak solution $u$ is semi-regular up to time $T>0$. 
\end{theorem}


\section{Preliminaries and proofs}
\subsection{Preliminaries}

In this section, we introduce some notations, and the notion of a
weak solution and the essential well-known result for the equations
\eqref{NNS}--\eqref{NNS-30}. For $1 \leq q \leq \infty $,  $W^{k,q}( \R^3 )$ indicates
the usual Sobolev space with standard norm $\norm{\cdot}_{k,q}$,
i.e., $W^{k,q}(\R^3) = \{ u \in L^{q}( \R^3 )\,:\, D^{ \alpha }u \in
L^{q}( \R^3 ), 0 \leq | \alpha | \leq k \}$. For vector fields $u,v$
we write $(u_iv_j)_{i,j=1,2,3}$ as $u\otimes v$. We denote by
$C=C(\alpha,\beta,...)$ a constant depending on the prescribed
quantities $\alpha,\beta,...$, which may change from line to line.


 Next we recall suitable weak solutions (localized weak solution) for
the the equations
\begin{definition}
Let $1<q<\infty$. The pair $(u,\pi)$ is called a suitable weak solution to the incompressible non-Newtonian fluids \eqref{NNS}--\eqref{NNS-30} on an open set
$D \subset \mathbb{R}^3 \times \mathbb{R}$, if the following conditions are met:

\begin{enumerate}
\item[(1)] \textit{(Integrability hypotheses)} $u,\pi$ are measurable functions on $D$ and
\begin{enumerate}
\item[(a)] $\pi \in L^{5/4}(D)$,
\item[(b)] for some constants $E_0,E_1<\infty$,
\[
\int_{\Omega} |u|^2 \, dx \le E_0,
\]
for almost every $t$, and
\[
\int_D |\nabla u|^p \, dx\,dt \le E_1 .
\]
\end{enumerate}

\item[(2)] \textit{(Equations)} $u,\pi$ satisfy \eqref{NNS}--\eqref{NNS-30} in the sense of distributions on $D$.

\item[(3)] \textit{(Generalized energy inequality)}  
For each real-valued function $\varphi \in C_0^\infty(D)$ with $\varphi \ge 0$,
the inequality
\begin{align}
&\int_{\Omega} |u|^2 \varphi
+ 2 \int_D \bigl( |\nabla u|^2 + |e(u)|^p \bigr) \varphi  \nonumber\\
&\qquad \le
\int_D |u|^2 (\partial_t \varphi + \Delta \varphi)-2\int_D |Du|^{q-2}Du:(u \otimes \nabla \varphi)+ 2 \int_D ( |u|^2 + 2\pi )\, u \cdot \nabla \varphi
\tag{2.1}
\end{align}
holds, where
\[
E_0(u)=\sup_{0\le t \le T} \int_{\Omega} |u|^2 \, dx,
\qquad
E_1(u)=\int_0^T \int_{\Omega} |\nabla u|^p \, dx\,dt .
\]
\end{enumerate}
\end{definition}


For suitable weak solutions to the the equations
\eqref{NNS}--\eqref{NNS-30}, it is possible to get
the following $\epsilon$-regularity theorem (see \cite{GZ02}).
\begin{proposition}\label{ep-reg}
Suppose that $u(x,t)$ is a suitable weak solution to
\eqref{NNS}--\eqref{NNS-30}  such that
\[
\limsup_{r\to 0} r^{2p-5} \int_{Q_r(x,t)} |\nabla u|^{p}\,dx\,dt \le \varepsilon,
\qquad p \leq \frac{5}{2},
\]
for some constant $\varepsilon$. Then
\[
|u(x,t)| \le C.
\]
That is, $(x,t)$ is a regular point.
\end{proposition}

Next, we see an weighted interpolation inequality in \cite{CKN} or
\cite[Lemma 2.1]{Yong09}.
\begin{lemma}\label{wip}
Let $n\ge1$ and let $1\le p,q<\infty$. There exists a constant $C>0$ such that for all
$u\in C_c^\infty(\mathbb{R}^n\setminus\{0\})$,
\begin{equation}\label{CKN:inequality}
\bigl\||x|^{-\gamma} u\bigr\|_{L^p(\mathbb{R}^n)}
\le
C
\bigl\||x|^{-\alpha}\nabla u\bigr\|_{L^2(\mathbb{R}^n)}^{\theta}
\bigl\||x|^{-\beta}u\bigr\|_{L^q(\mathbb{R}^n)}^{1-\theta}.
\end{equation}
if and only if the following relations hold:
\begin{equation}\label{CKN:scaling}
\frac1p + \frac{\alpha}{n}
=
\theta\!\left(\frac12 - \frac1n\right)
+
(1-\theta)\!\left(\frac1q + \frac{\beta}{n}\right),
\qquad 0\le\theta\le1,
\end{equation}
together with the admissibility conditions
\begin{equation}\label{CKN:admissible}
\alpha \le \theta,
\qquad
\beta \le 1-\theta,
\qquad
\gamma = \theta(\alpha-1) + (1-\theta)\beta.
\end{equation}
and $\alpha - 1 \le \gamma \le \alpha \quad \text{(if } p=q=2 \text{ case)}$
\end{lemma}

\begin{proof}
Multiplying the first equation of (1.1) by $|x- x_0|^{-1}u$\footnote{
Formally, the argument in this paper corresponds to testing the Navier--Stokes
equations with the vector field $|x-x_0|^{-1}u$.
Since the weight $|x-x_0|^{-1}$ is singular at $x=x_0$, this test function is not
admissible in the weak formulation.
The argument is therefore justified rigorously.

Let $\eta_\varepsilon\in C_c^\infty(\mathbb{R}^3)$ be a radial cutoff function
satisfying
\[
\eta_\varepsilon(x)=0 \quad \text{for } |x-x_0|<\varepsilon,
\qquad
\eta_\varepsilon(x)=1 \quad \text{for } |x-x_0|>2\varepsilon,
\]
and $|\nabla\eta_\varepsilon|\lesssim \varepsilon^{-1}$.
We define the regularized test function
\[
\phi_\varepsilon(x,t)
=
\eta_\varepsilon(x)\,|x-x_0|^{-1}u(x,t),
\]
which is smooth and compactly supported away from $x_0$.
Hence $\phi_\varepsilon$ is an admissible test function in the weak formulation.}, and so integrating in
$\R^3$, it follows that
\[
\frac{1}{2}\frac{d}{dt}\int_{\R^3}|x- x_0|^{-1}|u|^2
dx+\int_{\R^3}|Du|^q|x- x_0|^{-1}\,dx
\]
\begin{equation}\label{main-es-1}
\lesssim \Big|\int_{\R^3}|u|^3\nabla |x- x_0|^{-1}\,dx\Big|+\Big|\int_{\R^3}\Phi u \nabla
\cdot |x- x_0|^{-1}u\,dx\Big|+\Big|\int_{\R^3}|Du|^{q-1} : \nabla (|x- x_0|^{-1})|u|\,dx\Big|,
\end{equation}
where we use the integration by parts and the divergence free
conditions of $u$. The last term in \eqref{main-es-1}, it is estimated 
\[
\int_{\R^3}|Du|^{q-1} : \nabla (|x- x_0|^{-1})|u|\, dx
\]
\[
=\int_{\R^3}|Du|^{q-1}|x- x_0|^{\frac{q-1}{q}}|x- x_0|^{-\frac{q-1}{q}} : \nabla (|x- x_0|^{-1})|u|\, dx
\]
\[
\lesssim \int_{\R^3}|Du|^{q-1}|x- x_0|^{-\frac{q-1}{q}} |x- x_0|^{-1-\frac{1}{q}}|u|\, dx
\]
\[
\lesssim \Big(\int_{\R^3}|\nabla u|^{q}|x- x_0|^{-1}\,dx\Big)^{\frac{q-1}{q}} \Big(\int_{\R^3}|u|^{q}|x- x_0|^{-q-1}\,dx\Big)^{\frac{1}{q}}
\]
\[
\leq \frac{1}{16}\int_{\R^3}|\nabla u|^{q}|x- x_0|^{-1}\,dx+C_q\int_{\R^3}|u|^{q}|x- x_0|^{-q-1}\,dx
\]
\[
\leq\frac{1}{16}\int_{\R^3}|\nabla u|^{q}|x- x_0|^{-1}\,dx +C_q\||x-x_0|^{-1}| u|\|^{q}_{L^2}.
\]
Here, the following estimates is used: for $ q> 2$
\begin{align*}
&\int_{\mathbb{R}^3} |u|^{q} |x-x_0|^{-q-1} \, dx = \left\| |x-x_0|^{-1-\frac{1}{q}} u \right\|_{L^q}^q \\
&\quad \lesssim \left\| |x-x_0|^{-\frac{1}{2}} \nabla u \right\|_{L^q}^{q\theta} \left\| |x-x_0|^{-1} u \right\|_{L^2}^{q(1-\theta)}, \quad \text{with } \theta = \frac{3q-4}{4q-6} \\
&\quad \leq \frac{1}{16} \left\| |x-x_0|^{-\frac{1}{2}} \nabla u \right\|_{L^q}^{q} + C_q \left\| |x-x_0|^{-1} u \right\|_{L^2}^{q}.
\end{align*}
The equality \eqref{main-es-1} becomes 
\[
\frac{d}{dt}\int_{\R^3}|x- x_0|^{-1}|u|^2
dx+\int_{\R^3}|\nabla u|^q|x- x_0|^{-1}\,dx
\]
\begin{equation}\label{main-es-2}
\lesssim \Big|\int_{\R^3}|u|^3\nabla |x- x_0|^{-1}\,dx\Big|+\Big|\int_{\R^3}\Phi u \nabla
\cdot |x- x_0|^{-1}u\,dx\Big|+\Big|\||x-x_0|^{-1}| u|\|^{q}_{L^2}\Big|
\end{equation}
\[
:=I_1+I_2+I_3.
\]
Next, it is estimated the
term $I_1$ by H\"{o}lder's inequality as follows:
\[
I_1\lesssim \int_{\R^3}\frac{|u||u|^2}{|x- x_0|^2}\lesssim\|x-
x_0|^{-\frac{4}{3}}|u|^2\|_{L^{\frac{3}{2}}}\|x-x_0|^{-\frac{2}{3}}u\|_{L^3}
\lesssim\|x- x_0|^{-\frac{2}{3}}u\|^3_{L^{3}}.
\]
For $I_2$, taking the $\rm div$ operator on both sides of the equation
\eqref{NNS}, the pressure equation is written by
\[
-\Delta \Phi = \text{div}\text{div}(u\otimes u+|Du|^{q-2}Du).
\]
Using the Stein's inequality in
\cite{Stein}, we obtain
\begin{equation}\label{stein}
\||x- x_0|^{\delta}\Phi\|_{L^{\gamma}}\lesssim(\|x
-x_0|^{\delta}|u|^2\|_{L^{\gamma}}+\|x
-x_0|^{\delta}|Du|^{q-1}\|_{L^{\gamma}}),
\end{equation} where
\[
 1 <\gamma < \infty \quad \text{and} \quad - \frac{3}{\gamma}  < \delta < \frac{3(\gamma-1)}{\gamma}.
\]
By the H\"{o}lder inequality with \eqref{stein} and Young's inequality, it follows that
\[
I_3\lesssim\int_{\R^3}\Phi u \nabla \cdot |x- x_0|^{-1}dx\lesssim \|x-
x_0|^{-\frac{4}{3}}\Phi\|_{L^{\frac{3}{2}}}\|x-x_0|^{-\frac{2}{3}}u\|_{L^3}
\]
\[
\lesssim \Big(\|x- x_0|^{-\frac{4}{3}}|u|^2\|_{L^{\frac{3}{2}}}+\|x-
x_0|^{-\frac{4}{3}}|Du|^{q-1}\|_{L^{\frac{3}{2}}}\Big)\|x-x_0|^{-\frac{2}{3}}u\|_{L^3}
\]
\begin{equation}\label{aa115}
\lesssim(\|x- x_0|^{-\frac{2}{3}}u\|^3_{L^{3}}+\|x-
x_0|^{-\frac{4}{3}}|\nabla u|^{q-1}\|^{\frac{3}{2}}_{L^{\frac{3}{2}}}).
\end{equation}
%

By Lemma \ref{wip} and Young's inequality, we observe that
\[
\||x- x_0|^{-\frac{2}{3}}u\|^3_{L^3}\lesssim\||x-x_0|^{-\frac{1}{q}}\nabla u\|^{3\theta_1}_{L^q}\||x
-x_0|^{\kappa}u\|^{3(1-\theta_1)}_{L^{\delta}}
\]
\begin{equation}\label{thm2-l3}
\leq \frac{1}{16}\|x -x_0|^{-\frac{1}{q}}\nabla u\|^q_{L^q} + C\||x
-x_0|^{\kappa}u\|^{3(1-\theta_1)\frac{q}{q-3\theta_1}}_{L^{\delta}}
\end{equation}
\begin{equation}\label{thm2-l3}
\leq \frac{1}{16}\|x -x_0|^{-\frac{1}{q}}\nabla u\|^q_{L^q} + C\||x
-x_0|^{-\frac{1}{2}}u\|^{2}_{L^2}\||x
-x_0|^{\beta}|\nabla u|\|^{s}_{L^{r}},
\end{equation}
where
\[
\frac{1}{3}+\frac{-\frac{2}{3}}{3}=\theta_1\Big(\frac{1}{q}+\frac{-\frac{1}{q}-1}{3}\Big)+(1-\theta_1)\Big(\frac{1}{\delta}+\frac{\kappa}{3}\Big),
\]
\[
\frac{1}{\delta}+\frac{\kappa}{3}=\theta_2\Big(\frac{1}{2}+\frac{-\frac{1}{2}}{3}\Big)+(1-\theta_2)\Big(\frac{1}{r}+\frac{\beta-1}{3}\Big),
\]and
\[
2+s=3(1-\theta_1)\frac{q}{q-3\theta_1}\quad 3\theta_2(1-\theta_1)\frac{q}{q-3\theta_1}=2.
\]
From the relations, it knows
\[
\theta_1=\frac{qs-q}{3(2+s-q)},\quad \theta_2=\frac{2(3-q)}{6+(3-q)s-2q},\quad \frac{1}{\delta}
+\frac{\kappa}{3}=\frac{18-9q+5sq-8s}{6(6+(3-q)s-2q)}.
\]
Through Lemma \ref{wip}, it also knows
\[
-\frac{2}{3}=\theta_1 \sigma_1+(1-\theta_1)\kappa,\quad \sigma_1\leq-\frac{1}{2},
\]
and
\[
\kappa =(1-\theta_2)\sigma_2 - \frac{1}{2}\theta_2,\quad \sigma_2 \leq\beta.
\]
 Equations (3.6) and (3.7) can be solved as
\[
\kappa\geq -\frac{qs-3s-3+q}{6+(3-q)s-2q},\quad \beta\geq -1.
\]
And thus, it choses $\delta$ and $\kappa$ as
\[
\delta=\frac{6(6+(3-q)s-2q)}{20-9q+5sq-8s},\quad\mbox{and}\quad \kappa= -\frac{1}{6+(3-q)s-2q}.
\]
Hence, we finally obtain with the aid of the inequality
\eqref{thm2-l3}
\begin{equation}\label{aaaaa}
\frac{d}{dt}\||x-
x_0|^{-\frac{1}{2}}|u|\|^2_{L^2(\R^3)}
+\int_{\R^3} |x- x_0|^{-1}|\nabla u|^qdx
\end{equation}
\[
\lesssim \||x
-x_0|^{-\frac{1}{2}}u\|^{2}_{L^2(\R^3)}\||x
-x_0|^{\beta}|\nabla u|\|^{s}_{L^{r}(\R^3)}+\|x-
x_0|^{-\frac{4}{3}}|\nabla u|^{q-1}\|^{\frac{3}{2}}_{L^{\frac{3}{2}}}+\||x-x_0|^{-1}| u|\|^{q}_{L^2}.
\]
Integrate both sides of the estimate \eqref{aaaaa} above on time, after then,
taking the supremum for $x_0\in \R^3$, it follows
\[
\sup_{x_0\in R^3}\||x-
x_0|^{-\frac{1}{2}}|u(\cdot,\tau)|\|^2_{L^2(\R^3)}+\sup_{x_0\in R^3}\int_0^T\int_{\R^3}|x- x_0|^{-1}|\nabla u(\cdot, \tau)|^2 dx d\tau
\]
\[
 \lesssim\sup_{x_0\in R^3}\||x-
x_0|^{-\frac{1}{2}}|u_0|\|^2_{L^2(\R^3)}
\]
\[
+\sup_{x_0\in R^3}\int_0^T \||x-
x_0|^{\beta}u(\cdot,\tau)\|^s_{L^r(\R^3)}\||x-
x_0|^{-\frac{1}{2}}|u(\cdot,\tau)|\|^2_{L^2(\R^3)}\, d\tau
\]
\[
+\sup_{x_0\in R^3}\int_0^T \|x-
x_0|^{-\frac{4}{3}}|\nabla u(\cdot,\tau)|^{q-1}\|^{\frac{3}{2}}_{L^{\frac{3}{2}}}\, d\tau+\sup_{x_0\in R^3}\int_0^T\||x-x_0|^{-1}| u(\cdot,\tau)|\|^{q}_{L^2}, d\tau
\]
By the standard Gronwal's inequality for $q<2$, we obtain
\begin{equation}\label{u-result}
\sup_{x_0\in R^3}\int_0^T\int_{\R^3}|x- x_0|^{-1}|\nabla u|^2 dx dt<\infty.
\end{equation}
For $t > 0$ and $x \in \R^3$, due to the \eqref{u-result}, we have
\[
\limsup_{r\rightarrow0}\frac{1}{r^{5-2q}}\int_{t-r^2}^t\int_{B_r(x)}|\nabla
u(y,\tau)|^q\,dy\,d\tau
\leq
\limsup_{r\rightarrow0}\int_{t-r^2}^t\int_{\R^3}|y-x|^{-1}|\nabla
u(y,\tau)|^2\,dy\,d\tau=0.
\]
%
%

Finally, since $B_r(x) \subset \mathbb{R}^3$, we conclude that
\[
\frac{1}{r^{5-2q}}
\int_{t-r^2}^t
\int_{B_r(x)} |\nabla u|^q
\lesssim
\int_{t-r^2}^t
\int_{\mathbb{R}^3} |y-x|^{-1} |\nabla u|^2 .
\]

This implies the solution $u$ is regular for $0< t \leq T$, by
Proposition \ref{ep-reg}.

 
\end{proof}

\end{document}